\documentclass[11pt, reqno]{amsart}
\usepackage{textcmds} 
\usepackage[shortlabels]{enumitem}
\usepackage{amsmath, amssymb, amsfonts, amstext, verbatim, amsthm, mathrsfs}
\usepackage{microtype}
\usepackage[all]{xy}
\usepackage[modulo]{lineno}
\usepackage{aliascnt}
\usepackage{enumitem}
\usepackage{xspace}
\usepackage{amsfonts}
\usepackage{amssymb}
\usepackage[centertags]{amsmath}
\usepackage{amsthm}
\usepackage{rotating}
\usepackage[margin=3.25cm]{geometry}
\usepackage{dsfont}
\usepackage{bm}
\usepackage{subfigure}
\usepackage{amsmath}
\usepackage{array}
\usepackage[all]{xy}
\usepackage{euscript}
\usepackage[T1]{fontenc}
\usepackage{mathbbol}
\usepackage{nccmath}
\usepackage{marginnote}
\usepackage{stmaryrd}
\usepackage{youngtab}
\usepackage{tikz}
\usepackage{blkarray}
\usepackage{chngcntr}
 \usepackage[abs]{overpic}

\usepackage{tikz}

\usepackage{graphics,graphicx}  

\let\oldtocsection=\tocsection
\let\oldtocsubsection=\tocsubsection

\renewcommand{\tocsection}[2]{\hspace{0em}\oldtocsection{#1}{#2}}
\renewcommand{\tocsubsection}[2]{\hspace{1em}\oldtocsubsection{#1}{#2}}

\usepackage[colorlinks=true,linkcolor=black,citecolor=blue,urlcolor=blue,citebordercolor={0 0 1},urlbordercolor={0 0 1},linkbordercolor={0 0 1}]{hyperref} 

\tikzset{node distance=3cm, auto}

\makeatletter
\def\@secnumfont{\bfseries}
\def\section{\@startsection{section}{1}%
  \z@{.7\linespacing\@plus\linespacing}{.5\linespacing}%
  {\normalfont\Large\bfseries}}
\def\subsection{\@startsection{subsection}{2}%
  \z@{.75\linespacing\@plus.7\linespacing}{-.5em}%
  {\normalfont\large\bfseries}}

  \def\subsubsection{\@startsection{subsubsection}{3}%
  \z@{.75\linespacing\@plus.7\linespacing}{-.5em}%
  {\normalfont\bfseries}}
\makeatother

\newtheorem{thm}{Theorem}[subsection]
\newtheorem{lemma}[thm]{Lemma}
\newtheorem{prop}[thm]{Proposition}

\newtheorem{question}[thm]{Question}

\newtheorem{remark}[thm]{Remark}

\theoremstyle{remark}
\numberwithin{equation}{subsection} 
 \numberwithin{figure}{section}

\newtheoremstyle{customremark}
{3pt}
{3pt}
{}
{}
{\bfseries}
{.}
{.5em}
{}
\theoremstyle{customremark}
\newtheorem{rmk_no_diamond}[thm]{Remark}

\newtheorem{example_no_diamond}[thm]{Example}

\newcommand{\er}{{\Diamond}}

\makeatletter
\newcommand{\dashover}[2][\mathop]{#1{\mathpalette\df@over{{\dashfill}{#2}}}}
\newcommand{\fillover}[2][\mathop]{#1{\mathpalette\df@over{{\solidfill}{#2}}}}
\newcommand{\df@over}[2]{\df@@over#1#2}
\newcommand\df@@over[3]{%
  \vbox{
    \offinterlineskip
    \ialign{##\cr
      #2{#1}\cr
      \noalign{\kern1pt}
      $\m@th#1#3$\cr
    }
  }%
}
\newcommand{\dashfill}[1]{%
  \kern-.5pt
  \xleaders\hbox{\kern.5pt\vrule height.4pt width \dash@width{#1}\kern.5pt}\hfill
  \kern-.5pt
}
\newcommand{\dash@width}[1]{%
  \ifx#1\displaystyle
    2pt
  \else
    \ifx#1\textstyle
      1.5pt
    \else
      \ifx#1\scriptstyle
        1.25pt
      \else
        \ifx#1\scriptscriptstyle
          1pt
        \fi
      \fi
    \fi
  \fi
}
\newcommand{\solidfill}[1]{\leaders\hrule\hfill}

\newcommand{\oset}[3][0ex]{%
  \mathrel{\mathop{#3}\limits^{
    \vbox to#1{\kern-2\ex@
    \hbox{$\scriptstyle#2$}\vss}}}}

\makeatother

\date{\today}

\title{The symplectic Hadamard question}

\author{Dan Cristofaro-Gardiner}
\address{Mathematics Department, University of Maryland, College Park}
\email{dcristof@umd.edu}
\thanks{DCG thanks the NSF for their support under agreements DMS-2227372 and DMS-2238091}
\subjclass{53D05}

\begin{document}

\maketitle

\begin{abstract}

We show that a simply connected symplectic manifold admitting a complete nonpositively curved compatible metric is symplectomorphic to $\mathbb{R}^{2n}$ with its standard symplectic form.
This answers the "symplectic Hadamard question" of McDuff-Salamon in the affirmative.

\end{abstract}

\section{Introduction}

Any symplectic manifold has a plethora of compatible Riemannian metrics, but the relationship between the Riemannian geometry of such a metric and the underlying symplectic geometry is still quite mysterious.  For example, the following beautiful question has attracted interest.  Recall that {\em Hadamard's theorem} states that a simply connected complete nonpositively curved manifold is diffeomorphic to $\mathbb{R}^n$.  It is then natural to ask:

\begin{question}\cite[Prob. 51]{McDuffSalamon2017} (``Symplectic Hadamard Question")
If a simply connected symplectic manifold admits a complete nonpositively curved compatible metric, must the manifold be symplectomorphic to $\mathbb{R}^{2n}$ with its standard symplectic form?
\end{question}

The aim of this note is to resolve this question.  More precisely, we prove the following:

\begin{thm}
\label{thm:main}
Let $(M,\omega,J,g)$ be a simply connected symplectic manifold with compatible almost complex structure inducing a metric $g$.   Assume that $g$ is complete and all sectional curvatures of $g$ are nonpositive.  Then $(M,\omega)$ is symplectomorphic to $(\mathbb{R}^{2n},\omega_{std})$.
\end{thm}

We recall that a {\em compatible almost complex structure} $J$ on a symplectic manifold is one such that $g(u,v) := \omega(u,Jv)$ defines a Riemannian metric.  A compatible triple $(\omega,J,g)$ is sometimes called an {\em almost K\"{a}hler} structure.   Previously, McDuff \cite{McDuff1988} had proved the above theorem in the K\"{a}hler case, i.e. when $J$ is integrable.  It will be helpful to observe in what follows that $J$ acts by isometries with respect to $g$, i.e. $g(u,v) = g(Ju,Jv),$ and that $\omega$ can be recovered from $g$ by the rule $\omega(u,v) = g(Ju,v).$ 

\subsection{On AI use}

This paper resulted from some of my attempts to experiment with artificial intelligence.   I have been interested in exploring the capabilities of current publicly available models and thought the ``symplectic Hadamard question" looked like an interesting potential use case, sitting as it does
at the intersection of symplectic, complex, and Riemannian geometry.
In this case, it took
 a back and forth over a few hours with ChatGPT 5.5 Pro to arrive at a proposed solution, 
which I then checked and refined over several days, also with the help of GPT.  The paper is entirely written by me.  You can see the transcript of the main research interactions at \url{https://chatgpt.com/share/6a55446f-f514-83ea-96cf-850912167bff}.   (I have not included the transcript of the aforementioned interactions around checking and refining the argument, which are quite long and less interesting.) 

The above transcript is rather long (and I was optimizing my prompts for speed of research discovery and my own convenience rather than for future sharing), so I feel I should say a few words about my impressions of the interaction.  As far as key contributions to the final product, in my view the crucial contribution of ChatGPT was the critical identity \eqref{eqn:levi}.   There are two proofs of this, a direct computation and a reduction to a previous reference in the literature, and GPT found both of these.  Given this, one can essentially just follow McDuff's argument, after a standard application of Hessian comparison.    As for my contributions, other than selecting the problem, I felt it took a fair amount of prodding before we got to \eqref{eqn:levi}, and I had to redirect GPT from an approach that I found not promising (the reader can see the start of this redirection at the prompt starting "Let's back up and try to work together a little bit more..."; I suggest also looking at the model output right above that prompt).  I also felt that I had to insist on the general structure of the argument.

\section{The proof}

\subsection{Plurisubharmonicity of the square distance function and the key estimate}

Let $f$ be any smooth function on $(M,\omega,J,g)$.  Recall the {\em Levi form}, which is central to the arguments of McDuff in the K\"{a}hler case.  We recall the definition: we regard $df \circ J$ as a one form, then consider the two form $d(df \circ J)$, and then the Levi form is given by 
\[ L_f(X,Y) = -d(df \circ J)(X, JY).\]
We recall that a function $f$ is called {\em strictly plurisubharmonic} if $L_f(X,X) > 0$ for $X \ne 0$.

Now fix a point $p \in M$.   Let $f = r^2$, where $r$ is the distance to $p$ as determined by the metric; this is smooth given our curvature assumption.  
We now claim that
the nonpositive curvature forces the squared distance function to be strictly plurisubharmonic, with quantitative control; see Lemma~\ref{lem:key} below.  This is the key fact that we will use to prove Theorem~\ref{thm:main}.  

To elaborate,
denote the two-form underlying the Levi form for $f = r^2$ by
\[ \Omega(X,Y) = - d(dr^2 \circ J)(X,Y).\]
In the K\"{a}hler case, McDuff estimates $\Omega(v,Jv)$ by appealing to rather general estimates of Greene-Wu and Siu-Yau \cite[Lem. 2.1]{McDuff1988}, \cite[Prop. 2.24]{GreeneWu1979} \cite{siuyau}.  This makes the generalization to the almost K\"{a}hler case seem potentially quite difficult.  However, it turns out that the needed estimate follows quite simply from an identity implicit (though not explicit) in work of 
Harvey-Lawson 
around their potential theory on almost complex manifolds, specifically their \cite[Thm. 9.1]{HarveyLawson2015}.  This was pointed out to the author by ChatGPT Pro 5.5 and came as a surprise.
Thus, one can completely bypass the Greene-Wu and Siu-Yau works for our purposes.  The claimed identity is the 
following:
\begin{equation}
\label{eqn:levi}
 \Omega(v,Jv) = \text{Hess}_{g}(r^2)(v,v) + \text{Hess}_g(r^2)(Jv,Jv).
 \end{equation}
We have stated this for $f = r^2$, because that is all we need to establish the theorem in the present work, but as will see from the proof the analogue of \eqref{eqn:levi} holds for any smooth function. We can think of \eqref{eqn:levi} as asserting that the quadratic form associated to the Levi form agrees with the quadratic form associated to the Hermitian symmetric part of the Riemannian Hessian; see also the discussion in \cite[Ch. 4]{HarveyLawson2015}. 
In the K\"{a}hler case, this is known, 
and a related identity was proved in the contact case \cite[Prop. 3.4, 3.7]{EKM2012Invent}, \cite[Prop.~3.2]{EKM2016TAMS}.  The identity allows us to generalize the McDuff estimate using the standard Hessian comparison theorem from e.g. Lee's textbook \cite[Thm.~11.7]{Lee2018}.
Namely, we have:

\begin{lemma}
\label{lem:key}
Under the hypotheses of Theorem~\ref{thm:main}, we have
\begin{equation}
\label{eqn:keyestimate}
 \Omega(v,Jv) \ge 4|v|^2.
 \end{equation}
\end{lemma}

\begin{proof}
This is a standard argument that appears for example in the proof of \cite[Lem. 12.15]{Lee2018}.  Since all curvatures are bounded from above by $0$, the Hessian comparison theorem \cite[Thm.~11.7]{Lee2018} says that the Hessians are no smaller than they would be in the Euclidean case.  For the standard Euclidean metric $g_0$ we would have $\text{Hess}_{g_0}(r^2)(v,v) = 2 | v|^2$, so we obtain the key
\begin{equation}
\label{eqn:keyestimate}
 \Omega(v,Jv) \ge 2|v|^2 + 2 |Jv|^2 = 4|v|^2.
 \end{equation}
 \end{proof}

\subsubsection{More about Harvey-Lawson's identity}

We now explain the translation between $\Omega$ and the work of Harvey-Lawson, so as to see how to extract the identity \eqref{eqn:levi} from their work.  

To do this, we introduce some notation and some context.
Harvey-Lawson recall
that on any almost complex manifold, there is a {\em complex Hessian} \cite[Defn. 2.3]{HarveyLawson2015} associated to any function $f$, $\mathcal{H}(f)(v,w) = \partial \overline{\partial}f(v,\overline{w})$, taking as input $(1,0)$ tangent vectors.
This has an associated {\em real form}, a Hermitian symmetric form on the real tangent space, which on pairs $(v,v)$ takes the form
\[ H(f)(v,v) :=  \mathcal{H}(f)(v-i Jv, v - i Jv) \]
\[   =   \partial \overline{\partial} f(v - i Jv, v + iJv) = 2i \partial \overline{\partial}f (v,Jv).\]
Now, by definition, $2i \partial \overline{\partial} (r^2) (v,Jv) = \Omega(v,Jv)$.  To see this, note that for any $f$, $\overline\partial f(v) = \frac{1}{2}df(v + i Jv)$, so that $2i d \overline{\partial}f = - d(df \circ J)$; on the other hand, the anti-invariant component
\[ \frac{1}{2} \left( - d(df \circ J)(v,w) + d(df \circ J)(Jv,Jw)\right)\]
of
$d(df \circ J)$, i.e. the sum of the $(2,0)$ and $(0,2)$ parts, vanishes on pairs of the form $(v,Jv)$, so $d(df \circ J)$ is equal to its projection to $(1,1)$-forms when evaluated on pairs $(v,Jv)$. 

To sum up the above dictionary, then,
\begin{equation}
\label{eqn:language}
 \Omega(v,Jv) = H(r^2)(v,v),
 \end{equation}
in the notation of Harvey-Lawson:

We can now explain \eqref{eqn:levi}.  The proof given below summarizes what we need from Harvey-Lawson's proof of their \cite[Thm. 9.1]{HarveyLawson2015}.

\begin{proof}[Proof of \eqref{eqn:levi}, via \cite{HarveyLawson2015}]
Fix $v, Jv$, nonzero tangent vectors at any point $q$.  By standard theory \cite{NijenhuisWoolf1963}, \cite[Thm. 2.13.2]{wendl} there is a $J$-holomorphic curve $\Sigma$ tangent to the line spanned by $v,Jv$, which by restricting the domain we can assume is embedded.  We choose an extension of $v$ tangent to this holomorphic curve, noting that $H(f)(v,v)$ does not depend on the choice of extension.  Now we note that Harvey-Lawson show \cite[Lem. 9.4]{HarveyLawson2015} that in any Hermitian almost complex manifold (not necessarily almost K\"{a}hler), we have
\[ H(f)(v,v) = \text{Hess}_g(f)(v,v) + \text{Hess}_g(f)(Jv,Jv) + H_\Sigma \cdot f |v|^2,\]
where $H_\Sigma$ is the mean curvature vector field.  Since we are in addition assuming that our triple is almost K\"{a}hler, by standard theory $\Sigma$ is $g$-minimal, see e.g. the exposition in  \cite{mcduff}, and so $H_\Sigma$ vanishes, hence
\eqref{eqn:levi}, in view of \eqref{eqn:language}.  
\end{proof}

\begin{remark}
One can give a short alternate proof of the needed identity \eqref{eqn:levi} via direct tensorial calculation and we present this in the appendix.
\end{remark}

\subsection{The rest of the proof}

Equipped with \eqref{eqn:keyestimate}, the rest of the proof essentially follows McDuff's original argument from \cite{McDuff1988}; there is a subtlety coming from the fact that there is no canonical metric associated to the Levi form but this is not too severe.   We give a self-contained proof inspired by McDuff's that uses some more modern language at points. 

\subsubsection{The Weinstein manifold}

Let us first note that $\Omega$ is symplectic, for example by \eqref{eqn:keyestimate}.  Now in this subsection we prove:

\begin{prop}
\label{prop:weinstein}
The manifold $(M,\Omega)$ is symplectomorphic to $(\mathbb{R}^{2n},\omega_{std})$.
\end{prop}

\begin{proof}

Write $\Omega = d \theta$, for $\theta = - dr^2 \circ J$, and define the vector field $Z$ by reflecting $\theta$ via $\Omega$.  We claim that $(M,\Omega,r^2,Z)$ is a Weinstein manifold.

To show that $Z$ is gradient like for $r^2$, we note that $Z$ is the gradient of $r^2$ with respect to the two tensor $\Omega^s(v,w) = \Omega(v,Jw)$; i.e. $\Omega^s(Z,v) = dr^2(v):$  to see this, write $\Omega^s(Z,v) = \Omega(Z,Jv) = \theta(Jv) = dr^2(v).$  Moreover, the     
two tensor $\Omega^s(v,w) = \Omega(v,Jw)$ is positive definite, by \eqref{eqn:keyestimate}.  Hence $Z$ is
gradient like by \cite[Rem. 9.10]{CieliebakEliashberg2012}.   

To show that $Z$ generates a flow for all time, we estimate its length from above as follows.  By \eqref{eqn:keyestimate} we have $\Omega(Z,JZ) \ge 4 |Z|^2$, hence $|Z|^2 \le \frac{1}{4}|\theta(JZ)| \le \frac{1}{4} |\theta||Z|$ and so we get
\[ |Z| \le \frac{1}{4} |\theta| = \frac{1}{2} r.\]
Now we observe that a vector field of length $\le r$ on a complete manifold generates a flow for all time; this is a standard argument using Gr\"{o}nwall's inequality.

It is now easy to see that the Morse function $r^2$ has a unique critical point, which is the minimum, and so as is well-known \cite[Cor.~11.21, Cor.~11.27]{CieliebakEliashberg2012} the underlying symplectic manifold $(M,\Omega)$ is symplectomorphic to $\mathbb{R}^{2n}$ with its standard symplectic form.

\end{proof}

\subsection{Moser deformation}

Thus it remains to show the following. 

\begin{prop}
\label{prop:moser}
$(M,\omega)$ is symplectomorphic to $(M,\Omega)$.
\end{prop}

\begin{proof}

We argue by
Moser's trick, namely we define the symplectomorphism by flowing for time $1$ along a vector field $v_t$.  To recall how this works, we consider the convex combination $\omega_t = t \Omega + (1-t) \omega$, noting that $\omega_t(v,Jv) \ge |v|^2$ by \eqref{eqn:keyestimate} and so $\omega_t$ is symplectic.  Thus we can find the vector field $v_t$ by reflecting $\lambda - \theta$ using $\omega_t$, where $\lambda$ is a primitive of $\omega$ to be specified in a moment, and the Moser argument, i.e. a computation using Cartan's magic formula, shows that the time $1$ flow will give the desired symplectomorphism as long as the flow is defined.

Thus we want to show that $v_t$ generates a flow, so as mentioned in the previous section, as $M$ is complete it suffices to show $|v_t| \le 4r$, by a standard argument using Gr\"{o}nwall's inequality.  To show this, it would suffice to show
\begin{equation}
\label{eqn:formbound}
 |\lambda - \theta| \le 4r,
 \end{equation}
 since $ |\lambda - \theta||v_t| \ge  \omega_t(v_t,Jv_t)  \ge |v_t|^2$ by \eqref{eqn:keyestimate}.  
 We choose the primitive $\lambda$, and then estimate it, as follows.  

Because the curvatures are nonpositive,
 radial contraction to $p$ defines a smooth homotopy $F_t$ on $M$ such that $F_0$ maps all of $M$ to $p$ and $F_1$ is the identity.  We now let $A$ be the standard homotopy operator on the de Rham complex associated to the homotopy $F_t$  and we set $\lambda = A \omega$: explicitly, this means that $\lambda$ is defined by viewing the homotopy as a map $M \times [0,1] \to M$, pulling back $\omega$, and then integrating along the fibres.  Thus $\lambda$ is a primitive of $\omega$ by the chain homotopy equation. 

 To estimate the size of $\lambda$, 
we compute that,
by the definition of  the homotopy operator $A$,  
 at a point $x \in M$ with distance $R$ from $p$, if $\gamma$ denotes the unit speed geodesic connecting $p$ to $x$ then  
\[ \lambda_x(v) = \int^1_0 \omega_{\gamma(t R)}( R \dot{\gamma}(tR), Q(t)) dt.\]
where $Q(t)$ is some Jacobi field along $\gamma$
such that 
\begin{equation}
\label{eqn:jacobi}
Q(1) = v, \quad \quad Q(0) = 0.
\end{equation}   
Because $g$ is compatible with $\omega$ we have that 
\[ | \omega_{\gamma(t R)}( R \dot{\gamma}(tR), Q(t))| \le | R \dot{\gamma}(tR) | | Q(t) | =   R  |Q(t)|.\]
As is standard, the size of Jacobi fields \cite[Ch.~IV, Lem.~2.3]{Ballmann1995} in manifolds of nonpositive curvature is convex along geodesics.  Hence 
by \eqref{eqn:jacobi},
$|Q(t)| \le t|v|$.  Putting it all together, we therefore obtain that 
\[ |\lambda_x(v)| \le \frac{1}{2} R |v|,\]
hence the desired estimate for $\lambda$.   Noting that $|\theta| = 2r$, we derive \eqref{eqn:formbound}.  
 
\end{proof}

\vspace{5 mm}

\begin{proof}[Proof of Theorem~\ref{thm:main}]
This is an immediate consequence of Proposition~\ref{prop:weinstein} and Proposition~\ref{prop:moser}.
\end{proof}

{\bf Acknowledgements: } I thank Mohammed Abouzaid, Shaoyun Bai, Tamas Darvas, Tony Feng, Jeremy Kahn, 
Rohil Prasad, and Rich Schwartz for helpful questions and comments.  I thank the NSF for their support under agreements DMS-
2227372 and DMS-2238091.

\section{Appendix: a tensorial derivation of \eqref{eqn:levi}}
\label{sec:AI}

The following derivation of \eqref{eqn:levi} was worked out with ChatGPT.  (ChatGPT verified this identity autonomously and together we simplified the argument.)  Though it seems to have some aspects in common with ideas from \cite{EKM2012Invent, EKM2016TAMS,HarveyLawson2015}, I have not been able to find this derivation in its exact form in the literature.  This derivation is in fact essentially the initial proof (rewritten and shortened) of \eqref{eqn:levi} provided to me by GPT in the midst of our discussions, though later we discovered the Harvey-Lawson reference after some further back and forth.

\begin{proof}[An alternate proof of \eqref{eqn:levi}.]

We start by observing that the one form $-df \circ J$ agrees with $\iota_X \omega$, where $X = \nabla f$.  Indeed, $\iota_X \omega(v) = \omega(X,v) = g(JX,v) = - g(X,Jv) = - df(Jv).$   Thus we need to evaluate $d \iota_X \omega$ and by the Cartan formula this is just $\mathcal{L}_X \omega$, since $\omega$ is closed.  We have therefore shown that in the case $f = r^2,$ we have that $\mathcal{L}_X \omega = \Omega$ and so we just have to understand this Lie derivative.

To evaluate $\mathcal{L}_X \omega,$ we first note the 
product rule
\[ (\mathcal{L}_X\omega)(A,B) = X (\omega(A,B)) - \omega([X,A],B) - \omega (A, [X,B])\]
when $A$ and $B$ are vector fields. Now extend $v$ to a smooth vector field which we will continue to denote by $v$, and substitute into the above equation to get
\[ (\mathcal{L}_X\omega)(v,Jv) = X( g(Jv,Jv) )  - g(J [X,v], Jv) - g( Jv, [X, Jv]).\]
\[= 2 g ( \nabla_X v, v) - g( [X,v], v) - g([X,Jv], Jv) .\]
\begin{equation}
\label{eqn:bracketed}
 = [g( \nabla_X v, v) - g(\nabla_X Jv, Jv)]  + g(\nabla_v X, v) + g( \nabla_{Jv} X, Jv).
 \end{equation}
We have 
\[ g( \nabla_v X, v) = \text{Hess}_g(f)(v,v), \quad g( \nabla_{Jv} X, Jv) = \text{Hess}_g(f)(Jv,Jv).\]
On the other hand, differentiating $g(v,v) = g(Jv,Jv)$ with respect to $X$ gives that $g(\nabla_X v, v) = g( \nabla_X (Jv), Jv)$, so the bracketed term in \eqref{eqn:bracketed} vanishes as well, hence the result.

\end{proof}


\begin{thebibliography}{99}

\bibitem{Ballmann1995}
W.~Ballmann,
\emph{Lectures on Spaces of Nonpositive Curvature},
OWS Seminar, vol.~25,
Birkh\"auser Basel, 1995.


\bibitem{CieliebakEliashberg2012}
K.~Cieliebak and Y.~Eliashberg,
\emph{From Stein to Weinstein and Back: Symplectic Geometry of Affine Complex Manifolds},
American Mathematical Society Colloquium Publications, vol.~59,
American Mathematical Society, Providence, RI, 2012.

\bibitem{EKM2012Invent}
J.~B. Etnyre, R.~Komendarczyk, and P.~Massot,
Tightness in contact metric 3-manifolds,
\emph{Invent. Math.} \textbf{188} (2012), no.~3, 621--657.
doi:10.1007/s00222-011-0355-2.

\bibitem{EKM2016TAMS}
J.~B. Etnyre, R.~Komendarczyk, and P.~Massot,
Quantitative Darboux theorems in contact geometry,
\emph{Trans. Amer. Math. Soc.} \textbf{368} (2016), no.~11, 7845--7881.
doi:10.1090/tran/6821.

\bibitem{GreeneWu1979}
R.~E. Greene and H.~Wu,
\emph{Function Theory on Manifolds Which Possess a Pole},
Lecture Notes in Mathematics, vol.~699,
Springer, Berlin, 1979.


\bibitem{HarveyLawson2015}
F.~R. Harvey and H.~B. Lawson, Jr.,
Potential theory on almost complex manifolds,
\emph{Ann. Inst. Fourier} \textbf{65} (2015), no.~1, 171--210.

\bibitem{Lee2018}
J.~M. Lee,
\emph{Introduction to Riemannian manifolds},
2nd edition.
Graduate Texts in Mathematics, vol.~176,
Springer Cham, 2018.

\bibitem{mcduff}
D.~McDuff, {\em Symplectic structures --- a new approach to geometry}, Notices of the AMS, 45.8, 1998.

\bibitem{McDuff1988}
D.~McDuff,
The symplectic structure of K\"ahler manifolds of nonpositive curvature,
\emph{J. Differential Geom.} \textbf{28} (1988), no.~3, 467--475.

\bibitem{McDuffSalamon2017}
D.~McDuff and D.~Salamon,
\emph{Introduction to Symplectic Topology},
3rd ed., Oxford, 2017.

\bibitem{wendl}
C. Wendl, {\em Lectures on Holomorphic Curves in
Symplectic and Contact Geometry}, version 3.3, available on the author's website.

\bibitem{NijenhuisWoolf1963}
A.~Nijenhuis and W.~B. Woolf,
Some integration problems in almost-complex and complex manifolds,
\emph{Ann. of Math. (2)} \textbf{77} (1963), 424--489.
doi:10.2307/1970126.

\bibitem{siuyau} Y.T. Siu and S.T. Yau, {\em Complete K\"{a}hler manifolds with nonpositive
curvature of faster than quadratic decay}, Ann. Math. 105
(1977), 225-264.  doi:10.2307/1970998.

\end{thebibliography}
\end{document}